\documentclass[a4paper,11pt,reqno]{amsart}

\usepackage{amsthm, mathrsfs,amssymb,amsmath,}
\usepackage{enumerate}
\usepackage{mathtools}
\usepackage{cite}
\makeatletter
\@namedef{subjclassname@2010}{%
	\textup{2010} Mathematics Subject Classification}
\makeatother

\usepackage{hyperref}
\allowdisplaybreaks
\numberwithin{equation}{section}

\usepackage{graphicx}
\usepackage{enumitem}

\newtheorem{theorem}{Theorem}[section]
\newtheorem{lemma}[theorem]{Lemma}

\newtheorem{Cor}[theorem]{Corollary}
\theoremstyle{definition}

\theoremstyle{remark}
\newtheorem{rem}{Remark}

\newenvironment{pff}{\hspace*{-\parindent}{\bf Proof \,}}
{\hfill $\Box$ \vspace*{0.2cm}}

\DeclareMathOperator*{\esssup}{ess\,sup}

\def\P{{\mathcal{P_{\textrm{aff}}}}}
\def\v{{\vartheta}}
\def\Z{{\mathbb Z}}

\def\R2{{\mathbb R}^{2}}

\def\F{{\mathcal F}}

\def\Z{{\mathbb Z}}

\def\R2n{{\mathbb R}^{2n}}
\def\R{{\mathbb R}}

\def\F{{\mathcal F}}

\title[Pseudo-Differential Operators, Wigner Transform and Weyl Transform on the Affine Poincar\'e Group]{Pseudo-Differential Operators, Wigner Transform and Weyl Transform on the Affine Poincar\'e Group}
\author{Aparajita Dasgupta}
\address{Aparajita Dasgupta \endgraf Department of Mathematics \endgraf Indian Institute of Technology Delhi \endgraf New Delhi - 110 016, India.}
\email{adasgupta@maths.iitd.ac.in}

\author[Santosh Kumar Nayak]{Santosh Kumar Nayak}
\address{
	Santosh Kumar Nayak:
	\endgraf
	Department of Mathematics
	\endgraf
	Indian Institute of Technology, Delhi, Hauz Khas
	\endgraf
	New Delhi-110016 
	\endgraf
	India}
\endgraf
\email{nayaksantosh212@gmail.com }

\usepackage{fancyhdr}
\begin{document}
	
	\thanks{The research of A. Dasgupta was supported by MATRICS Grant(RP03933G), Science and Engineering Research Board (SERB), DST, India and Santosh Kumar Nayak was supported by IIT-D Institute fellowship.
	}
	\thanks{Santosh Kumar Nayak gives special thanks to Shyam Swarup Mondal for his support during the preparation of paper.}
	\date{\today}
	\maketitle
	\begin{abstract}
		In this paper, we study harmonic analysis on the affine Poincar\'e group $\P$, which is a non-unimodular group, and obtain pseudo-differential operators with
	operator valued symbols. More precisely, we study the boundedness properties of pseudo-differential operators on $\P$. 
We  also provide a necessary and sufficient condition on the operator-valued symbols such that the corresponding pseudo-differential operators are in the class of Hilbert--Schmidt operators. Consequently, we obtain a characterization of  the trace class pseudo-differential operators on the Poincar\'e affine group $\P$,  and provide a trace formula for these trace class operators.  Finally, we study the Wigner transform, and Weyl transform associated with the operator valued symbol on the Poincar\'e affine group $\P$.

	\end{abstract}

	\section{Introduction}
	One of the most active branches in modern contemporary mathematics is the theory of pseudo-differential operators. These are the generalizations of partial differential operators, which originated as an essential tool in the study of partial differential equations, numerical analysis, and quantum theory. 
	 Let $\sigma:\R^n\times\widehat{\R^n}\to \mathbb{C}$ be a measurable function, where $\widehat{\R^n}$ is dual of $\R^n,$ the set of all inequivalent irreducible unitary representations. Then the global pseudo-differential operator $T_{\sigma}$ associated with $\sigma$ is defined by 
	\begin{equation}\label{pdo}
			(T_{\sigma}f)(x)=(2\pi)^{-n/2}\int_{\R^n}e^{ix\cdot\xi}\sigma(x,\xi)\widehat{f}(\xi)d\xi,\quad x\in\R^n,
	\end{equation}
 for all $f\in S(\R^n)$, the Schwartz space on $\R^n$, provided that the integral exists, where $\widehat{f}$ is the Euclidean Fourier transform of $f$. The basic tool to define a pseudo-differential operator is the Fourier inversion formula, which is defined by
	$$f(x)=(2\pi)^{-n/2}\int_{\R^2}e^{ix\cdot\xi}\widehat{f}(\xi)d\xi,\quad x\in\R^n,$$ in which we insert a suitable symbol $\sigma$ on the phase space $\mathbb{R}^n\times
	\widehat{\mathbb{R}^n}$  in the Fourier inversion formula to get \eqref{pdo}. These ideas have been implemented to construct the pseudo-differential operators on different groups, such as $\mathbb{Z},\mathbb{S}^1$, locally compact abelian groups, compact Lie groups, Heisenberg groups, type I non-unimodular locally compact groups, type I unimodular locally compact groups, affine group, polar affine group, Heisenberg motion group by several authors,\cite{Car1,Car2,Car3,Daku, Del1, Del2,Del3,Del4,Ku1,Ku2,Ku3,Ku4,mantoiu2,mantoiu1,dasgupta,San,San1,Shahla,Ruzha,wong1,wong2}. 
	
 The theory of pseudo-differential operators was initiated by Kohn and Nirenberg \cite{Kohn} on $\R^n$. After that H\"ormander \cite{Hor} carried out all properties of boundedness and compactness of these operators. 
	 
For unimodular group $G$, the Fourier transform of $f\in L^1(G)$ can be defined by 
	 \begin{equation}\label{F}
	 	\widehat{f}(\xi)=\int_{G}f(x)\xi(x)^{\ast}dx, \quad\xi\in\widehat{G},
	 \end{equation} where $\widehat{G}$ is the unitary dual of G, the set of all inequivalent classes of irreducible unitary representations. When $G$ is non-unimodular, the situation becomes more complicated, i.e., if we consider \eqref{F} as Fourier transform in this case, the Fourier inversion and  Plancherel formula would be disturbed. To eradicate these complications, we need to introduce an operator in \eqref{F} called the Duflo-Moore operator \cite{duflo}, which is a self-adjoint, positive unbounded operator. Considerable attention has been devoted to the study of pseudo-differential operators on the non-unimodular groups. For example, see \cite{dasgupta,San} for pseudo-differential operators on affine group, \cite{Shahla} for Poincar\'e unit disk, \cite{San1} for similitude group, and in general \cite{mantoiu1,mantoiu2} for non-unimodular type I group.
 
 The well-known results in the operator theory are the boundedness and compactness of linear operators. If $\sigma$ in $L^2(\R^n\times\widehat{\R^n})$, then the pseudo-differential operator $T_{\sigma}$ in \eqref{pdo} is a bounded linear operator from $L^2(\R^n)$ into $L^2(\R^n)$. In particular, the resulting operator $T_{\sigma}$ is in Hilbert-Schmidt class as explained in \cite{wong2}. Similar kinds of results were proved in \cite{dasgupta1} for pseudo-differential operators on the Heisenberg group, and for the  $H$-type group in \cite{Yin}. Recently, a necessary and sufficient condition on the symbols such that the corresponding  pseudo-differential operators on the affine group and similitude group (polar affine group) are in Hilbert-Schmidt class has been obtained in \cite{San,San1}. Though the similitude group is the natural generalization of the affine group, any group of the type $G=\mathbb{R}^2\rtimes H,$ where $H$ is a group consisting of $2\times 2$ matrices is also worth studying. This group $H$ which is acting of $\mathbb{R}^2$ generates an open free orbit. To be more specific, if the group $H$ is such that for some fixed $2-$vector $\vec{x},$ the set,
 $$\mathcal{O}_x=\{\vec{y}=h^{T}\vec{x}: h\in H\},$$ is an open set in $\mathbb{R}^2$ and for all $\vec{x}\neq 0,$ $h\vec{x}=\vec{x}$ if and only if $h$ is the identity matrix, then such a group $G$ has square interable representations and hence can be used to study wavelets. An example of such group is the affine Poinc\'{a}re group and it is a semidirect product of above type. It differs from the $SIM(2)$ group in that the spatial rotations $r_{\theta}$ are replaced by hyperbolic rotations. The underlying manifold is $2$-dimensional, and this group arises in the study of $2$-dimensional continuous wavelet transform in the book \cite{jean}. 
 
 Motivated by these and previous studies, in \cite{San1,San}, we study and extend some of the aforementioned results to the setting of   affine Poincar\'e group $\P$. As already noted, the situation here is similar to that of the similitude group, however, we have hyperbolic rotations now and not rigid rotations of space. This group is useful in problems involving the detection of extremely fast moving objects (such as occurs, for example, in high energy physical experiments).
 
 The paper is organized as follows, In Section \ref{s2}, we study the group structure of the affine Poincar\'e group. We investigate harmonic analysis for the group $\P$ in Section \ref{s3}. The construction of pseudo-differential operators, and their boundedness are discussed in Section \ref{s4}. We  also  obtain a necessary and sufficient conditions on the symbol $\tau$ such that the corresponding pseudo-differential operator $T_\tau$ on $\P$ is a  Hilbert--Schmidt operator. We present a characterization  for the  trace class pseudo-differential operators on  $\P$,  and  find their trace formula.  Finally, we introduce the Weyl transform associated with Wigner transform for affine Poincar\'e group $\P$ in Section \ref{s5}.
 
 \section{The affine Poincar\'e Goup}\label{s2}
 
 The affine Poincar\'e group $\P$ is a generalization of the affine group, in fact, it is a complexification of the affine group. This group arises in the study of $2$-dimensional wavelet transform. Similar to four parameter similitude group, $\P$ contains the translations $b$ in the image plane $\mathbb{R}^2$, global dilations (zooming in and out by $a>0$), but it has hyperbolic rotations around the origin ($\v\in \R$). The action on the plane is given by
	$$x=(b,a,\v)y= a\Lambda_{\v}y+b,$$ where
	$b\in \mathbb{R}^2$, $a>0$, and $\Lambda_{\v}$ is the $2\times 2$ hyperbolic rotation matrix
	\begin{equation}
		\Lambda_{\v}=\left(\begin{matrix}
			\cosh\v &  \sinh\v\\
			\sinh\v   &  \cosh\v
		\end{matrix}\right).
	\end{equation}
	A convenient representation of the joint transformation $(b,a,\v)$ is in the form of $3\times 3$ matrices
	\begin{equation}
		(b,a,\v)=\left(\begin{matrix}
			a\Lambda_\v &  b\\
			0^{T}   &  1
		\end{matrix}\right),~~~ 0^{T}=(0,0).
	\end{equation}
Then the matrix multiplications gives the composition of successive transformations and thus the group law is derived as
	\begin{eqnarray}
		(b,a,\v)\ast (b^{\prime},a^{\prime},\v^{\prime})&=&(b+a\Lambda_{\v} b^{\prime},aa^{\prime},\v+\v^{\prime}),\nonumber
	\end{eqnarray}
	With respect to the operation $\ast$, $\P$ is a non-abelian group in which $({0},1,0)$ is the identity element and $(\frac{-1}{a}\Lambda_{-\v}{b},\frac{1}{a},-\v)$ is the inverse of $({b},a,\v)$ in $\P$. Also, it can be shown that $\P$ is a non-unimodular group as its left and right Haar measures   $$d\mu_{L}({b},a,\v)=\dfrac{d{b}dad\v}{a^3},~~ d\mu_{R}({b},a,\v)=\dfrac{d{b}dad\v}{a},$$ respectively are different, and hence the modular function is given by $\Delta(b,a,\v)=\dfrac{1}{a^2}$.
	Moreover, the affine Poincar\'e group $\P$ has the structure of a semi-direct product:
	$$\P=\mathbb{R}^{2}\rtimes(\mathbb{R}_{\ast}^{+}\times \mathrm{SO}(1,1)),$$ where $\mathbb{R}^2$ is the subgroup of the translations, $\mathbb{R}^{\ast}_{+}$ that of dilations, and $\mathrm{SO}(1,1)$ of hyperbolic rotations. Topologically, one can write
	$\P=\R^2\times\mathcal{C}$, where $\mathcal{C}$ is any one of the four cones:
	\begin{eqnarray}
		C_{1}^1=\{x\in\R^2:x_1^2>x_2^2, +x_1>0\} \\
		C_{2}^1=\{x\in\R^2:x_1^2>x_2^2,- x_1>0\}\\
		C_{1}^2=\{x\in\R^2:x_1^2<x_2^2,+ x_1>0\}\\
		C_{2}^2=\{x\in\R^2:x_1^2<x_2^2,- x_1>0\}.
	\end{eqnarray}
	 Let us define the Fourier transform $\F$ and inverse Fourier transform $\F^{-1}$,  by 
	 \begin{eqnarray}
	 	(\mathcal{F}\varphi)(\xi)=\frac{1}{2\pi}\int_{\R^2}e^{i\langle\xi;x\rangle}\varphi(x)dx \label{MinkowskiF}\\
	 	(\mathcal{F}^{-1}\varphi)(x)=\frac{1}{2\pi}\int_{\R^2}e^{-i\langle\xi;x\rangle}(\F{\varphi})(\xi)d\xi ,\label{inverseMinkowskiF} 
	 \end{eqnarray}   
for all $\varphi\in S(\R^n)$,	where $\langle x;y\rangle=x_1y_1-x_2y_2$ is the  Minkowski inner product. In this article, we denote $\F$ and $\F^{-1}$ are the Minkowski-Fourier transform and inverse Minkowski-Fourier transform on $\R^2$, respectively. 
\begin{rem}
	The relation between Euclidean Fourier transform and Minkowski-Fourier transform on $\R^2$ is $$(\mathcal{F}\varphi)(\xi_{1},\xi_2)=\widehat{\varphi}(-\xi_1,\xi_2), (\F^{-1}\varphi)(\xi_1,\xi_2)=\check{\varphi}(-\xi_1,\xi_2),$$ where $\widehat{\varphi}$ and $\check{\varphi}$ are the Euclidean Fourier transform  and inverse Euclidean Fourier transform $\varphi$ on $\R^2$, respectively. The reader can easily verify that the Minkowski-Fourier transform carries almost all properties similar to Euclidean Fourier transform on $\R^2$, like Plancherel formula, Parseval identity.  
\end{rem}
\section{Harmonic Analysis on the affine Poincar\'e Group}\label{s3}
Let us denote $L^2(C_{i}^j), i,j=1,2$, be the set of all square integrable of complex-valued functions on $C_{i}^j\subset\P$.	
	Define the mappings $\pi_i^j:\P\to U\big(L^2(C_i^j)\big), i,j=1,2$ by
	$$(\pi_{i}^j(b,a,\v)\varphi)(x)=ae^{i\langle x;b \rangle}\varphi(a\Lambda_{-\v}x), i,j=1,2,$$ for all $(b,a,\v)\in\P$ and $\varphi\in L^2(C_{i}^{j})$,  where $U\big(L^2(C_i^j)\big)$ is the set of all unitary operators on $L^2(C_i^j), i,j=1,2.$
	\begin{theorem}
		$\{\pi_i^j:i,j=1,2\}$ are the only infinite dimensional, irreducible and unitary representation of the affine Poincar\'e group $\P$.
	\end{theorem}
\begin{pff}
A brief discussion of this proof is available in \cite{jean}.	
\end{pff}

	Let $\varphi\in L^2(C_{i}^{j})$. We define the operators 
	\begin{equation}\label{duflo}
	(D_{i}^j\varphi)(x)=\frac{1}{2\pi}|\langle x;x\rangle|^{1/2}\varphi(x),
	\end{equation}
	for $x\in C_i^j$, where $i,j=1,2$. These linear operators $D_i^j, i,j=1,2,$ are unbounded, self-adjoint operators and are known as Duflo-Moore operators. For a detailed study on Duflo-Moore operators, we refer to \cite{duflo}. This Duflo-Moore operators are used to define the Fourier transform and to obtain the Plancherel theorem for the non-unimodular group $\P$.

	 For $f\in L^2(\P)\cap L^1(\P)$, we define the Fourier transform of $f$ on $L^2(C_{i}^j)$ by
	 $$(\widehat{f}(\pi_{i}^j)\varphi)(x)=\int_{\P}f(b,a,\v)\big(\pi_{i}^{j}(b,a,\v)(D_{i}^j\varphi)\big)(x)d\mu(b,a,\v),$$ for all $\varphi\in L^2(C_i^j), i,j=1,2$.
	 \begin{theorem}\label{kernelF}
	 	Let $f\in L^1(\P)\cap L^2(\P)$. Then $\widehat{f}(\pi_i^j)$ are Hilbert-Schmidt operator on $L^2(C_i^j)$ with the kernel given by
	 	\begin{equation}\label{kernel}
	 	 K_{i,j}^f(x,y)=(\mathcal{F}_1f)\Bigg(x,\Bigg(\dfrac{\langle y;y\rangle}{\langle x;x\rangle}\Bigg)^{\frac{1}{2}}, \cosh^{-1}\Bigg(\dfrac{\langle x;y\rangle}{ \langle x;x\rangle}\times \Bigg(\dfrac{\langle x;x\rangle}{\langle y;y\rangle}\Bigg)^{\frac{1}{2}}\Bigg)|\langle x;x\rangle|^{\frac{1}{2}}\dfrac{1}{|\langle y,y\rangle|},
	 		\end{equation}	
	for all $(x,y)\in C_i^j\times C_i^j$, $i,j=1,2$, where $\F_1$ denotes the Minkowski-Fourier transform of $f$ with respect to the first variable.
	 	
	 \end{theorem}
 \begin{pff}
 		Let $\varphi\in L^2(C_i^j),i,j=1,2$. Then for all $i,j=1,2$, using \eqref{MinkowskiF},  we have
 	\begin{align*}
 		(\widehat{f}(\pi_{i}^j)\varphi)(x) &=\int\limits_{\P}f(b,a,\v)\big(\pi_{i}^{j}(b,a,\v)(D_{i}^j\varphi)\big)(x)\frac{dbdad\v}{a^3}\\
 		&=\int\limits_{\P}f(b,a,\v)ae^{i\langle x;b\rangle}(D_{i}^j\phi)(a\Lambda_{-\v}x)\frac{dbdad\v}{a^3}\\
 		&= \frac{1}{2\pi}\int\limits_{\P}f(b,a,\v)ae^{i\langle x;b\rangle}|\langle a\Lambda_{-\v}x;a\Lambda_{-\v}x\rangle|^{\frac{1}{2}}\varphi(a\Lambda_{-\v}x)\frac{dbdad\v}{a^3}\\
 		&= \frac{1}{2\pi}\int\limits_{\P}f(b,a,\v)ae^{i\langle x;b\rangle}|\langle x;x\rangle|^{\frac{1}{2}}\varphi(a\Lambda_{-\v}x)\frac{dbdad\v}{a}\\
 		& = \int_0^{\infty}\int_{\R}(\mathcal{F}_1f)(x,a,\v)|\langle x;x\rangle|^{\frac{1}{2}}\varphi(a\Lambda_{-\v}x)\frac{dad\v}{a}. 
 	\end{align*}
 	Let $a\Lambda_{-\v}x=y$, then $a=\Bigg(\dfrac{\langle y;y\rangle}{\langle x;x\rangle}\Bigg)^{\frac{1}{2}}$,   $\v=\cosh^{-1}\Bigg(\dfrac{\langle x;y\rangle}{ \langle x;x\rangle}\times \Bigg(\dfrac{\langle x;x\rangle}{\langle y;y\rangle}\Bigg)^{\frac{1}{2}}\Bigg)$ and $dy=a|\langle x;x\rangle|dad\v.$ Hence
 	$$(\widehat{f}(\pi_{i}^j)\varphi)(x)=\int_{C_{i}^j}(\mathcal{F}_1f)\Bigg(x,\Bigg(\dfrac{\langle y;y\rangle}{\langle x;x\rangle}\Bigg)^{\frac{1}{2}}, \cosh^{-1}\Bigg(\dfrac{\langle x;y\rangle}{ \langle x;x\rangle}\times \Bigg(\dfrac{\langle x;x\rangle}{\langle y;y\rangle}\Bigg)^{\frac{1}{2}}\Bigg)|\langle x;x\rangle|^{\frac{1}{2}}\varphi(y)\dfrac{dy}{|\langle y,y\rangle|}.$$
 	Thus the kernel of $\widehat{f}(\pi_i^j)$ is
 	\begin{equation}
 	 K_{i,j}^f(x,y)=(\mathcal{F}_1f)\Bigg(x,\Bigg(\dfrac{\langle y;y\rangle}{\langle x;x\rangle}\Bigg)^{\frac{1}{2}}, \cosh^{-1}\Bigg(\dfrac{\langle x;y\rangle}{ \langle x;x\rangle}\times \Bigg(\dfrac{\langle x;x\rangle}{\langle y;y\rangle}\Bigg)^{\frac{1}{2}}\Bigg)|\langle x;x\rangle|^{\frac{1}{2}}\dfrac{1}{|\langle y,y\rangle|},
 		\end{equation}
 for all $(x,y)\in C_i^j\times C_i^j$, $i,j=1,2$. 
 	Now using the Plancherel formula for Minkowski-Fourier transform $\F$, we get
 	\begin{align*}
 		&	\int_{C_{i}^j}\int_{C_{i}^j}|K_{i,j}^f(x,y)|^2dxdy\\
 		&=\int_{C_{i}^j}\int_{C_{i}^j}\left|{(\mathcal{F}_1f)\Bigg(x,\Bigg(\dfrac{\langle y;y\rangle}{\langle x;x\rangle}\Bigg)^{\frac{1}{2}}, \cosh^{-1}\Bigg(\dfrac{\langle x;y\rangle}{ \langle x;x\rangle}\times \Bigg(\dfrac{\langle x;x\rangle}{\langle y;y\rangle}\Bigg)^{\frac{1}{2}}\Bigg)}\right|^2\dfrac{|\langle x;x\rangle|}{|\langle y;y\rangle|^2}dxdy\\
 		&=\int_{C_{i}^j}\int_0^{\infty}\int_{\R}\left| (\mathcal{F}_1f)(x,a,\v)\right|^2\dfrac{dxdad\v}{a^3}\\
 		&\leq \int_{\R^2}\int_0^{\infty}\int_{\R}\left| (\mathcal{F}_1f)(x,a,\v)\right|^2\dfrac{dxdad\v}{a^3}\\
 		&=\int_{\R^2}\int_0^{\infty}\int_{\R}\left| f(x,a,\v)\right|^2\dfrac{dxdad\v}{a^3}<\infty.
 	\end{align*}
 	
 \end{pff}


In the next lemma, we show $D_{i}^j\widehat{f}(\pi_{i}^j)\pi_{i}^j(b,a,\v)^{\ast},i,j=1,2$ are integral operators.
	\begin{lemma}
	Let $f\in L^1(\P)\cap L^2(\P)$. Then	$D_{i}^j\widehat{f}(\pi_{i}^j)\pi_{i}^j(b,a,\v)^{\ast}$ are integral operators on $L^2(C_i^j)$ with kernel
		$$S^{f,\pi_{i}^j}(x,y)=\dfrac{a}{2\pi}|\langle x;x\rangle|^{\frac{1}{2}}K_{i,j}^f(x,a\Lambda_{-\v}y)e^{-i\langle y;b\rangle},$$ for all $(x,y)\in C_i^j\times C_i^j,$ where $K_{i,j}^f$ is defined in \eqref{kernel}, $i,j=1,2$.  
	\end{lemma}
	\begin{pff}
		Using the Theorem \ref{kernelF} and Duflo-Moore operator in \eqref{duflo}, we get
		\begin{align*}
			&(D_{i}^j\widehat{f}(\pi_{i}^j)\pi_{i}^j(b,a,\v)^{\ast}\varphi)(x)\\
			 &= \frac{1}{2\pi}|\langle x;x\rangle|^{\frac{1}{2}}(\widehat{f}(\pi_{i}^j)\pi_{i}^j(b,a,\v)^{\ast}\varphi)(x)\\
			& = \frac{1}{2\pi}|\langle x;x\rangle|^{\frac{1}{2}}\int_{C_{i}^j}K_{i,j}^f(x,y)(\pi_{i}^j(b,a,\v)^{\ast}\varphi)(y)dy\\
			& = \frac{1}{2\pi}|\langle x;x\rangle|^{\frac{1}{2}}\int_{C_{i}^j}K_{i,j}^f(x,a\Lambda_{-\v}y)e^{-i\langle y;b\rangle}\varphi(y)ady.
		\end{align*}
	Hence the kernel of $D_{i}^j\widehat{f}(\pi_{i}^j)\pi_{i}^j(b,a,\v)^{\ast}$ is given by
	$$S^{f,\pi_{i}^j}(x,y)=\dfrac{a}{2\pi}|\langle x;x\rangle|^{\frac{1}{2}}K_{i,j}^f(x,a\Lambda_{-\v}y)e^{-i\langle y;b\rangle}, i,j=1,2.$$
	\end{pff}

Now we are ready to get the Fourier inversion formula. 
\begin{theorem}[Fourier Inversion Formula]
	
	Let $f\in L^1(\P)\cap L^2(\P)$, then
	\begin{equation}		f(b,a,\v)=\Delta(b,a,\v)^{-\frac{1}{2}}\sum_{i,j=1}^2\text{Tr}\Big[D_{i}^j\widehat{f}(\pi_{i}^j)\pi_{i}^j(b,a,\v)^{\ast}\Big], 
	\end{equation}
for all $(b,a,\v)\in \P.$
\end{theorem}
\begin{pff}
	We note that
	\begin{align}\label{S}
		S^{f,\pi_{i}^j}(x,x) &=\dfrac{a}{2\pi}|\langle x;x\rangle|^{\frac{1}{2}}K_{i,j}^f(x,a\Lambda_{-\v}x)e^{-i\langle x;b\rangle}\nonumber\\
		&=\dfrac{1}{2\pi a}(\mathcal{F}_1f)(x,a,\v)e^{-i\langle x;b\rangle}\nonumber\\
		&= \dfrac{1}{2\pi}\Delta(b,a,\v)^{\frac{1}{2}}(\mathcal{F}_1f)(x,a,\v)e^{-i\langle x;b\rangle},	
	\end{align}
	for all $x\in C_i^j$, $i,j=1,2.$
	
	Now using the equation \eqref{S} and inverse Minkowski-Fourier transform \eqref{inverseMinkowskiF}, we obtain
	\begin{align*}
		&\Delta(b,a,\v)^{-\frac{1}{2}}\sum_{i,j=1}^2\text{Tr}\Big(D_{i}^j\widehat{f}(\pi_{i}^j)\pi_{i}^j(b,a,\v)^{\ast}\Big)\\
		&=\Delta(b,a,\v)^{-1/2}\Bigg[\int_{C_{1}^1}S^{f,\pi_{1}^1}(x,x)dx+\int_{C_{2}^1}S^{f,\pi_{2}^1}(x,x)dx\\
		&+\int_{C_{1}^2}S^{f,\pi_{1}^2}(x,x)dx+\int_{C_{2}^2}S^{f,\pi_{2}^2}(x,x)dx\Bigg]\\
		&=\dfrac{1}{2\pi}\int_{\R^2}(\mathcal{F}_1f)(x,a,\v)e^{-i\langle x;b\rangle}dx\\
		&= f(b,a,\v),
	\end{align*}
 for all $(b,a,\v)$ in $\P$.
\end{pff}

Let $S_2(C_i^j)$ be the set of all Hilbert-Schmidt operators on $L^2(C_i^j), i,j=1,2$. Let us denote $$S_2=S_2(C_1^1)\cup S_2(C_2^1)\cup S_2(C_2^2)\cup S_2(C_1^2).$$
Let $L^2(\widehat{\P},S_2)$ be the space of all functions $F:\widehat{\P}\to S_2(\R^2)$ such that $$F(\pi_i^j)\in S_2(C_i^j),i,j=1,2.$$ The reader can easily check that $L^2(\widehat{\P},S_2)$ is a Hilbert space with the inner product
\begin{equation}\label{Plancherel}
	\langle F,G\rangle_{L^2(\widehat{\P},S_2)}=\text{Tr}\big(F(\pi_1^1)G(\pi_1^1)^{\ast}\big)+\text{Tr}\big(F(\pi_1^2)G(\pi_1^2)^{\ast}\big)+\text{Tr}\big(F(\pi_2^1)G(\pi_2^1)^{\ast}\big)+\text{Tr}\big(F(\pi_2^2)G(\pi_2^2)^{\ast}\big),
\end{equation} and norm
$$\|F\|_{L^2(\widehat{\P},S_2)}=\|F(\pi_1^1)\|_{S_2(C_1^1)}+\|F(\pi_2^1)\|_{S_2(C_2^1)}+\|F(\pi_1^2)\|_{S_2(C_1^2)}+\|F(\pi_2^2)\|_{S_2(C_2^2)},$$ for all $F,G\in L^2(\widehat{\P},S_2).$
For $f\in L^2(\P)$, we define the bounded linear operator $\widehat{f}(\pi):L^2(\R^2)\rightarrow L^2(\R^2)$ by
$$\widehat{f}(\pi)\varphi=\widehat{f}(\pi_{1}^1)\varphi_{1}^1+\widehat{f}(\pi_{2}^1)\varphi_{2}^1+\widehat{f}(\pi_{1}^2)\varphi_{1}^2+\widehat{f}(\pi_{2}^2)\varphi_{2}^2,$$ where $\varphi_{i}^j=\varphi\chi_{C_{i}^j}, i,j=1,2.$
%
\begin{theorem}[Plancherel Formula]\label{plancherel}
	Let $f\in L^2(\P)$. Then for all $\varphi\in L^2(\R^2)$, the operator $$(\widehat{f}(\pi)\varphi)(x)=\int_{\R^2}K^f(x,y)\varphi(y)dy, x\in\R^2,$$	where 
	\[ K^f(x,y)=
	\begin{cases}
		K_{1,1}^f(x,y), (x,y)\in C_{1}^1\times C_{1}^{1},\\
		K_{2,1}^f(x,y), (x,y)\in C_{2}^1\times C_{2}^{1},\\
		K_{1,2}^f(x,y), (x,y)\in C_{1}^2\times C_{1}^{2},\\
		K_{2,2}^f(x,y), (x,y)\in C_{2}^2\times C_{2}^{2},\\
		0,\quad\quad\quad\quad \text{otherwise},
	\end{cases}
	\]
	and $K_{i,j}^f,i,j=1,2$, are defined in \eqref{kernel}, is a Hilbert-Schmidt operator on $L^2(\R^2)$ and moreover,
	$$\|\widehat{f}\|_{L^2(\widehat{\P},S_2)}=\|\widehat{f}(\pi_1^1)\|^2_{S_2}+\|\widehat{f}(\pi_2^1)\|^2_{S_2}+\|\widehat{f}(\pi_1^2)\|^2_{S_2}+\|\widehat{f}(\pi_2^2)\|^2_{S_2}=\|f\|^2_{L^2(\P)}.$$
\end{theorem}
\begin{proof}
	Using the Plancherel formula for Minkowski-Fourier transform, we get
	\begin{align*}
		& \int_{\R^2}\int_{\R^2}\left| K^f(x,y)\right|^2dxdy\\
		&= \int_{C_{1}^1}\int_{C_{1}^1}|K_{1,1}^f(x,y)|^2dxdy+\int_{C_{2}^1}\int_{C_{2}^1}|K_{2,1}^f(x,y)|^2dxdy\\
		&+\int_{C_{1}^2}\int_{C_{1}^2}|K_{1,2}^f(x,y)|^2dxdy+\int_{C_{2}^2}\int_{C_{2}^2}|K_{2,2}^f(x,y)|^2dxdy\\
		&= \int_{C_{1}^1}\int_{0}^{\infty}\int_{\R}\left|(\mathcal{F}_1f)(x,a,\v)\right|^2\dfrac{dxdad\v}{a^3}+\int_{C_{2}^1}\int_{0}^{\infty}\int_{\R}\left|(\mathcal{F}_1f)(x,a,\v)\right|^2\dfrac{dxdad\v}{a^3}\\
		&+\int_{C_{1}^2}\int_{0}^{\infty}\int_{\R}\left|(\mathcal{F}_1f)(x,a,\v)\right|^2\dfrac{dxdad\v}{a^3}+\int_{C_{2}^2}\int_{0}^{\infty}\int_{\R}\left|(\mathcal{F}_1f)(x,a,\v)\right|^2\dfrac{dxdad\v}{a^3}\\
		& =\int_{\R^2}\int_0^{\infty}\int_{\R}\left|(\mathcal{F}_1f)(x,a,\v)\right|^2\dfrac{dxdad\v}{a^3}\\
		& = \int_{\R^2}\int_0^{\infty}\int_{\R}\left|f(x,a,\v)\right|^2\dfrac{dxdad\v}{a^3}=\|f\|^2_{L^2(\P)}.
	\end{align*}
Note that, we can write $\|f\|_{L^2(\P)}$ as $$\|\widehat{f}(\pi_1^1)\|^2_{S_2}+\|\widehat{f}(\pi_2^1)\|^2_{S_2}+\|\widehat{f}(\pi_1^2)\|^2_{S_2}+\|\widehat{f}(\pi_2^2)\|^2_{S_2}=\|f\|^2_{L^2(\P)}.$$
\end{proof}

Let $\sigma$ be a suitable measurable function on $\R^2 \times \R^2$. Then the Weyl transform $W_{\sigma}$ corresponding to the symbol $\sigma$ is defined by
$$\langle W_{\sigma}f,g\rangle_{L^2(\R^2)}=(2\pi)^{-1}\int_{\R^2}\int_{\R^2}\sigma(x,\xi)W(f,g)(x,\xi)dxd\xi,$$	for all $f$ and $g$ in $L^2(\R^2)$, where $W(f,g)$ is the Wigner transform of $f$ and $g$, defined as
$$W(f,g)(x,\xi)=(2\pi)^{-1}\int_{\R^2}e^{-i\xi\cdot q}f\Big(x+\frac{q}{2}\Big)g\Big(x-\frac{q}{2}\Big)dq.$$
 The following corollary is about relation between the group Fourier transform $\widehat{f}(\pi)$ of $f$ at $\pi$ and Weyl transform $W_{\sigma}$.
\begin{Cor}\label{WeylF}
Let $f\in L^2(\P)$. Then $\widehat{f}(\pi)=W_{\sigma_{f}},$ where $$\sigma_{f}(x,\xi)=2\pi(\F_2TK^f)(x,\tilde{\xi}),\quad \tilde{\xi}=(-\xi_1,\xi_2),$$ and $T:L^2(\R^2)\to L^2(\R^2)$ is an unitary twisting operator defined by $$(T\phi)(t_1,t_2)=\phi\Big(t_1+\frac{t_2}{2},t_1-\frac{t_2}{2}\Big),$$	and $\F_2$ denotes Minkowski-Fourier transform  in the second variable. 
\end{Cor}
\begin{pff}
	The proof of the corollary can be found in \cite{wong1}.
\end{pff}
	\section{Pseudo-differential Operators on the Affine Poincar\'e group}\label{s4}
Let $\sigma:\P\times\widehat{\P}\rightarrow B(L^2(C_{i}^1))\cup B(L^2(C_{j}^2))$ such that $$\sigma(b,a,\v,\pi_{i}^j)\in B(L^2(C_{i}^j)), i,j=1,2,$$ for all $(b,a,\v)\in\P$, where $B(L^2(C_i^j))$ are the set of all bounded linear operators from $L^2(C_i^j)$ to $L^2(C_i^j), i,j=1,2$. 
We define the pseudo-differential operator $T_{\sigma}:L^2(\P)\to L^2(\P)$ corresponding to the symbol $\sigma$  on $\P$ by
\begin{equation}\label{pseudo}
	(T_{\sigma}f)(b,a,\v)=\Delta(b,a,\v)^{-1/2}\sum_{i,j=1}^2\text{Tr}\Big[D_{i}^j\sigma(b,a,\v,\pi_{i}^j)\widehat{f}(\pi_{i}^j)\pi_{i}^j(b,a,\v)^{\ast}\Big], (b,a,\v)\in\P.	
\end{equation}
This also can be write
$$(T_{\sigma}f)(b,a,\v)=\int_{\P}K(b,a,\v,b^{\prime},a^{\prime},\v^{\prime})f(b^{\prime},a^{\prime},\v^{\prime})d\mu(b,a,\v),$$
where the kernel $K$ is given by
$$K(b,a,\v,b^{\prime},a^{\prime},\v^{\prime})=\frac{1}{a}\sum_{i,j=1}^2\text{Tr}\Big[D_{i}^j\sigma(b,a,\v,\pi_{i}^j)\pi_{i}^j(b^{\prime},a^{\prime},\v^{\prime})D_{i}^j\pi_{i}^j(b,a,\v)^{\ast}\Big].$$
Let us denote $S_p(C_{i}^j)$ be the $p$-Schatten class, set of all compact operators on the Hilbert space $L^2(C_{i}^j)$ whose singular values are $p$-th summable, $i,j=1,2,$ for $1\leq p< \infty$ and $S_{\infty}(C_i^j)=B(L^2(C_i^j))$, the set of all bounded operators on $L^2(C_i^j), i,j=1,2.$ The following theorem is $L^2-L^2$ boundedness of pseudo-differential operator $T_{\sigma}$ on $L^2(\P)$.
\begin{theorem}
	Let $\sigma:\P\times\widehat{\P}\to B(L^2(C_{i}^1))\cup B(L^2(C_{j}^2))$ be an operator valued symbol such that $D_{i}^j\sigma(b,a,\v,\pi_{i}^j)\in S_2(C_{i}^j)$, for all $(b,a,\v)\in\P$. Let $G_{i}^j:\P\to\mathbb{C}$ be the mappings defined by $$G_{i}^j(b,a,\v)=\Delta(b,a,\v)^{-1/2}\|K_{\sigma(b,a,\v,\pi_{i}^j)}\|_{L^2(C_{i}^{j}\times C_{i}^j)}, i,j=1,2,$$ is in $L^2(\P)$, where $K_{\sigma(b,a,\v,\pi_{i}^j)}$ is the kernel of the integral operator $D_{i}^j\sigma(b,a,\v,\pi_{i}^j), i,j=1,2$. Then $T_{\sigma}:L^2(\P)\to L^2(\P)$ is a bounded operator and moreover
	$$\|T_{\sigma}\|_{B(L^2(\P))}\leq \|G_1^1\|_{L^2(\P)}+\|G_2^1\|_{L^2(\P)}+\|G_1^2\|_{L^2(\P)}+\|G_2^2\|_{L^2(\P)}.$$
\end{theorem}
\begin{pff}
	Let $f\in L^2(\P)$. Then using the Plancherel formula \eqref{plancherel} and equation \eqref{pseudo}, we obtain
	\begin{align*}
		&\|T_{\sigma}f\|_{L^2(\P)}\\
		&=\Big(\int_{\P}|(T_{\sigma}f)(b,a,\v)|^2d\mu(b,a,\v)\Big)^{1/2}\\
		&=\Big(\int_{\P}\Delta(b,a,\v)^{-1}\left|\sum_{i,j=1}^2\text{Tr}\Big(D_{i}^j\sigma(b,a,\v,\pi_{i}^j)\widehat{f}(\pi_{i}^j)\pi_{i}^j(b,a,\v)^{\ast}\Big)\right|^2d\mu(b,a,\v)\Big)^{1/2}\\
		&\leq\Big(\int_{\P}\Delta(b,a,\v)^{-1}\sum_{i,j=1}^2\left|\text{Tr}\Big(D_{i}^j\sigma(b,a,\v,\pi_{i}^j)\widehat{f}(\pi_{i}^j)\pi_{i}^j(b,a,\v)^{\ast}\Big)\right|^2d\mu(b,a,\v)\Big)^{1/2}\\
		&\leq  \Big(\int_{\P}\Delta(b,a,\v)^{-1}\sum_{i,j=1}^2\|\widehat{f}(\pi_i^j)\|_{S_2(C_i^j)}^2\|D_i^j\sigma(b,a,\v,\pi_i^j)\|_{S_2(C_i^j)}^2d\mu(b,a,\v)\Big)^{1/2}\\
		&\leq  \sum_{i,j=1}^2\|\widehat{f}(\pi_i^j)\|_{S_2(C_i^j)}\Big(\int_{\P}\Delta(b,a,\v)^{-1}\|D_i^j\sigma(b,a,\v,\pi_i^j)\|_{S_2(C_i^j)}^2d\mu(b,a,\v)\Big)^{1/2}\\
		&=  \sum_{i,j=1}^2\|\widehat{f}(\pi_i^j)\|_{S_2(C_i^j)}\|G_i^j\|_{L^2(\P)}\\
		&\leq  \Big(\sum_{i,j=1}^2\|\widehat{f}(\pi_i^j)\|^2_{S_2(C_i^j)}\Big)^{1/2}\Big(\sum_{i,j=1}^2\|G_i^j\|_{L^2(\P)}^2\Big)^{1/2}\\
		&= \|f\|_{L^2(\P)}\Big(\sum_{i,j=1}^2\|G_i^j\|_{L^2(\P)}^2\Big)^{1/2}.
	\end{align*}
Thus $T_{\sigma}$ is bounded operator from $L^2(\P)$ into $L^2(\P)$ and 
$$\|T_{\sigma}\|_{B(L^2(\P))}\leq \|G_1^1\|_{L^2(\P)}+\|G_2^1\|_{L^2(\P)}+\|G_1^2\|_{L^2(\P)}+\|G_2^2\|_{L^2(\P)}.$$
\end{pff}
\begin{Cor}
	If $D_{i}^j\sigma(b,a,\v,\pi_{i}^j)\in S_p(C_{i}^j),i,j=1,2, 1\leq p<2,$ then $T_{\sigma}:L^2(\P)\to L^2(\P)$ is a bounded linear operator. 
\end{Cor}
We obtain $L^2-L^p$ boundedness of the pseudo-differential operator $T_{\sigma}$ from $L^2(\P)$ to $L^p(\P)$.
\begin{theorem}Let $2\leq p\leq\infty$ and $q$ be its conjugate index $(\frac{1}{p}+\frac{1}{q}=1).$
	Let $\sigma:\P\times\widehat{\P}\to B(L^2(C_{i}^1))\cup B(L^2(C_{j}^2))$ be an operator valued symbol such that $D_{i}^j\sigma(b,a,\v,\pi_{i}^j)\in S_q(C_{i}^j)$, for all $(b,a,\v)\in\P$, $i,j=1,2$. Let $G_{i}^j:\P\to\mathbb{C}$ be the mappings defined by $$G_{i}^j(b,a,\v)=\Delta(b,a,\v)^{-1/2}\|D_{i}^j\sigma(b,a,\v,\pi_{i}^j)\|_{S_q(C_{i}^j)}, $$ is in $L^p(\P),  i,j=1,2$. Then $T_{\sigma}:L^2(\P)\to L^p(\P)$ is a bounded operator, and
	$$\|T_{\sigma}\|_{B(L^2(\P),L^p(\P))}\leq \left(\|G_1^1\|_{L^p(\P)}+\|G_2^1\|_{L^p(\P)}+\|G_1^2\|_{L^p(\P)}+\|G_2^2\|_{L^p(\P)}\right).$$
\end{theorem}
\begin{pff}
	Let $f\in L^2(\P)$, and $2\leq p<\infty$. Then by the H\"older inequality, Plancherel formula \eqref{plancherel} and equation \eqref{pseudo}, we get
	\begin{align}\label{pbounded}
		&\|T_{\sigma}f\|_{L^p(\P)}\nonumber\\
		=&\Big(\int_{\P}|(T_{\sigma}f)(b,a,\v)|^pd\mu(b,a,\v)\Big)^{1/p}\nonumber\\
		=&\Big(\int_{\P}\Delta(b,a,\v)^{-p/2}\left|\sum_{i,j=1}^2\text{Tr}\Big(D_{i}^j\sigma(b,a,\v,\pi_{i}^j)\widehat{f}(\pi_{i}^j)\pi_{i}^j(b,a,\v)^{\ast}\Big)\right|^pd\mu(b,a,\v)\Big)^{1/p}\nonumber\\
		\leq &\Big(\int_{\P}\Delta(b,a,\v)^{-p/2}\sum_{i,j=1}^2\left|\text{Tr}\Big(D_{i}^j\sigma(b,a,\v,\pi_{i}^j)\widehat{f}(\pi_{i}^j)\pi_{i}^j(b,a,\v)^{\ast}\Big)\right|^pd\mu(b,a,\v)\Big)^{1/p}\nonumber\\
		\leq & \Big(\int_{\P}\Delta(b,a,\v)^{-p/2}\sum_{i,j=1}^2\|\widehat{f}(\pi_i^j)\|_{S_p(C_i^j)}^p\|D_i^j\sigma(b,a,\v,\pi_i^j)\|_{S_q(C_i^j)}^pd\mu(b,a,\v)\Big)^{1/p}\nonumber\\
		\leq & \sum_{i,j=1}^2\|\widehat{f}(\pi_i^j)\|_{S_p(C_i^j)}\Big(\int_{\P}\Delta(b,a,\v)^{-p/2}\|D_i^j\sigma(b,a,\v,\pi_i^j)\|_{S_q(C_i^j)}^pd\mu(b,a,\v)\Big)^{1/p}\nonumber\\
		\leq & \sum_{i,j=1}^2\|\widehat{f}(\pi_i^j)\|_{S_2(C_i^j)}\|G_i^j\|_{L^p(\P)}\nonumber\\
		\leq & \Big(\sum_{i,j=1}^2\|\widehat{f}(\pi_i^j)\|^2_{S_2(C_i^j)}\Big)^{1/2}\Big(\sum_{i,j=1}^2\|G_i^j\|_{L^p(\P)}^2\Big)^{1/2}\nonumber\\
		=& \|f\|_{L^2(\P)}\Big(\sum_{i,j=1}^2\|G_i^j\|_{L^p(\P)}^2\Big)^{1/2}.
	\end{align}
For $p=\infty$, using the fact $S_2(C_i^j)\subset S_{\infty}(C_i^j), i,j=1,2,$ and H\"older inequality, we get
\begin{align}\label{infintybounded}
	&\|T_{\sigma}f\|_{L^{\infty}(\P)}=\esssup\limits_{(b,a,\v)\in\P}|T_{\sigma}f(b,a,\v)|\nonumber\\
	&=\esssup\limits_{(b,a,\v)\in\P}\left|\Delta(b,a,\v)^{-1/2}\sum_{i,j=1}^2\text{Tr}\Big[D_{i}^j\sigma(b,a,\v,\pi_{i}^j)\widehat{f}(\pi_{i}^j)\pi_{i}^j(b,a,\v)^{\ast}\Big]\right|\nonumber\\
	&\leq\esssup\limits_{(b,a,\v)\in\P}\Delta(b,a,\v)^{-1/2}\sum_{i,j=1}^2\left|\text{Tr}\Big[D_{i}^j\sigma(b,a,\v,\pi_{i}^j)\widehat{f}(\pi_{i}^j)\pi_{i}^j(b,a,\v)^{\ast}\Big]\right|\nonumber\\
	&\leq\esssup\limits_{(b,a,\v)\in\P}\Delta(b,a,\v)^{-1/2}\sum_{i,j=1}^2\|D_i^j\sigma(b,a,\v,\pi_i^j)\|_{S_1(C_i^j)}\|\widehat{f}(\pi_i^j)\|_{S_{\infty}(C_i^j)}\nonumber\\
	&=\esssup\limits_{(b,a,\v)\in\P}\sum_{i,j=1}^2|G_i^j(b,a,\v)| \|\widehat{f}(\pi_i^j)\|_{S_{\infty}(C_i^j)}\nonumber\\
	&=\sum_{i,j=1}^2\|G_i^j\|_{L^{\infty}(\P)}\|\widehat{f}(\pi_i^j)\|_{S_{\infty}(C_i^j)}\nonumber\\
	&\leq \sum_{i,j=1}^2\|G_i^j\|_{L^{\infty}(\P)}\|\widehat{f}(\pi_i^j)\|_{S_2(C_i^j)}\nonumber\\
	&\leq \Big(\sum_{i,j=1}^2\|G_i^j\|_{L^{\infty}(\P)}^2\Big)^{1/2}\Big(\sum_{i,j=1}^2\|\widehat{f}(\pi_i^j)\|_{S_2(C_i^j)}^2\Big)^{1/2}\nonumber\\
	&= \|f\|_{L^2(\P)}\Big(\sum_{i,j=1}^2\|G_i^j\|_{L^{\infty}(\P)}^2\Big)^{1/2}.
\end{align}
From \eqref{pbounded} and \eqref{infintybounded}, we have
 $$\|T_{\sigma}f\|_{L^p(\P)}\leq \|f\|_{L^2(\P)}\Big(\sum_{i,j=1}^2\|G_i^j\|_{L^{p}(\P)}^2\Big)^{1/2}, \quad 2\leq p\leq\infty.$$
\end{pff}

\begin{theorem}[Uniqueness Theorem]\label{uniqueness1}
	
Let $\sigma :{\P}\times\{ \pi_{i}^{j} \} \to S_2$ be an operator valued symbol such that $$\sum_{i,j=1}^2\int_{\P} \|\sigma(b,a,\v,\pi_{i}^j)\|_{S_2}^2\dfrac{dbdad\v}{a^3}<\infty ,$$ and the mapping
\begin{equation}\label{weakly}
	\P\times\widehat{\P}\owns (b,a,\v,\pi_i^j)\mapsto \pi_i^j(b,a,\v)^{\ast}D_i^j\sigma(b,a,\v,\pi_i^j)\in S_2
\end{equation} is weakly continuous. Then $T_{\sigma}f=0$ for all $f\in L^2(\P)$ iff $\sigma(b,a,\v,\pi_i^j)=0$ for almost all $(b,a,\v)$ in $\P$.	
\end{theorem}
\begin{pff}
	Let the symbol $\sigma$ be zero. Then by the definition of pseudo-differential operator \eqref{pseudo}, $T_{\sigma}f=0$ for all $f$ in $L^2(\P)$.
	
	 Conversely, suppose that $(T_{\sigma}f)(b,a,\v)=0$ for all $f\in L^2(\P)$. Let $(b^{\prime},a^{\prime},\v^{\prime})$ in $\P$. Then we define the function $f_{b^{\prime},a^{\prime},\v^{\prime}}$ in $L^2(\P)$ such that $$\widehat{f_{b^{\prime},a^{\prime},\v^{\prime}}}(\pi_i^j)=\big(D_i^j\sigma(b^{\prime},a^{\prime},\v^{\prime},\pi_i^j)\big)^{\ast}\pi_i^j(b^{\prime},a^{\prime},\v^{\prime}).$$ Then the pseudo-differential operator becomes
	\begin{equation}\label{uniqueness}
		(T_{\sigma}f_{b^{\prime},a^{\prime},\v^{\prime}})(b,a,\v)
		=\frac{1}{a}\sum_{i,j=1}^2 \text{Tr}\Big[D_i^j\sigma(b,a,\v,\pi_i^j)(D_i^j\sigma(b^{\prime},a^{\prime},\v^{\prime},\pi_i^j))^{\ast}\pi_i^j(b^{\prime},a^{\prime},\v^{\prime})\pi_i^j(b,a,\v)^{\ast}\Big].
	\end{equation} Let $(b_0,a_0,\v_0)$ in $\P$. Then by weak continuity of the mapping \eqref{weakly}, we get
	\begin{align*}
		&\text{Tr}\Big[D_i^j\sigma(b,a,\v,\pi_i^j)(D_i^j\sigma(b^{\prime},a^{\prime},\v^{\prime},\pi_i^j))^{\ast}\pi_i^j(b^{\prime},a^{\prime},\v^{\prime})\pi_i^j(b,a,\v)^{\ast}\Big]\\
		&\to \text{Tr}\Big[D_i^j\sigma(b_0,a_0,\v_0,\pi_i^j)(D_i^j\sigma(b^{\prime},a^{\prime},\v^{\prime},\pi_i^j))^{\ast}\pi_i^j(b^{\prime},a^{\prime},\v^{\prime})\pi_i^j(b_0,a_0,\v_0)^{\ast}\Big],
	\end{align*}
as $(b,a,\v)\to (b_0,a_0,\v_0)$.
Hence the equation \eqref{uniqueness} becomes 
$$T_{\sigma}f_{b^{\prime},a^{\prime},\v^{\prime}}(b,a,\v)\to T_{\sigma}f_{b^{\prime},a^{\prime},\v^{\prime}}(b_0,a_0,\v_0),$$ as $(b,a,\v)\to (b_0,a_0,\v_0)$, and hence $T_{\sigma}f_{b^{\prime},a^{\prime},\v^{\prime}}(b,a,\v)$ is a continuous function on $\P$. Letting $(b_0,a_0,\v_0)=(b^{\prime},a^{\prime},\v^{\prime})$, we obtain
\begin{align*}
	T_{\sigma}f_{b^{\prime},a^{\prime},\v^{\prime}}(b^{\prime},a^{\prime},\v^{\prime})&= \dfrac{1}{a^2}\sum\limits_{i,j=1}^2\|D_i^j\sigma(b^{\prime},a^{\prime},\v^{\prime},\pi_i^j)\|_{S_2(C_i^j)}=0.
\end{align*}
Therefore $\|D_i^j\sigma(b^{\prime},a^{\prime},\v^{\prime},\pi_i^j)\|_{S_2(C_i^j)}=0,$ for all $i,j=1,2,$. Since $D_i^j$ are positive operators, hence $\sigma(b^{\prime},a^{\prime},\v^{\prime},\pi_i^j)=0$ for almost $(b^{\prime},a^{\prime},\v^{\prime})\in \P.$
\end{pff}

Using the Theorem \ref{plancherel} and corollary \ref{WeylF}, we can define  $\widehat{f}(\pi_i^j)=W_{\sigma_f}^{i,j}, \quad i,j=1,2.$

The following theorem characterizes Hilbert-Schmidt pseudo-differential operators on $L^2(\P)$ under some suitable operator valued symbol.
\begin{theorem}[Hilbert-Schmidt pseudo-differential Operator]\label{Hilbert-Schmidt}
	Let $\sigma:\P\times\P\to S_2$ be an operator valued symbol such that the hypotheses of Theorem \ref{uniqueness1} is satisfied. Then the associated pseudo-differential operator $T_{\sigma}:L^2(\P)\to L^2(\P)$ is a Hilbert-Schmidt operator iff
	$$D_i^j\sigma(b,a,\v,\pi_i^j)\pi_i^j(b,a,\v)^{\ast}=W_{\tau_{\alpha(b,a,\v)}},\quad ,i,j=1,2,$$ for all $(b,a,\v)\in\P$, where 
	$$\tau_{\alpha(b,a,\v)}(x,\xi)=\F_2^{-1}TK_{\alpha(b,a,\v)}(x,\tilde{\xi}),\quad\tilde{\xi}=(-\xi_1,\xi_2),$$ $T$ is a twisting operator, 
$$K_{\alpha(b,a,\v)}(x,{\xi})=\dfrac{a}{2\pi}\times\dfrac{|\langle x,x\rangle|^{\frac{1}{2}}}{|\langle \xi;\xi\rangle|}\times (\F_1^{-1}\alpha(b,a,\v))\Bigg(x,\Bigg(\dfrac{\langle \xi;\xi\rangle}{\langle x;x\rangle}\Bigg)^{\frac{1}{2}}, \cosh^{-1}\Bigg(\dfrac{\langle x;\xi\rangle}{ \langle x;x\rangle}\times \Bigg(\dfrac{\langle x;x\rangle}{\langle \xi;\xi\rangle}\Bigg)^{\frac{1}{2}}\Bigg),$$ for all $(x,\xi)\in C_i^j\times C_i^j$, $i,j=1,2$, and $\alpha:\P\to L^2(\P)$ is the weakly continuous map such that $$\int_P\|\alpha(b,a,\v)\|_{L^2(\P)}^2\dfrac{dbdad\v}{a^3}<\infty.$$
 	
\end{theorem}
\begin{pff} 
We know that the pseudo-differential operator $T_{\sigma}$ on $C_0^{\infty}(\P)$ is defined by
\begin{equation}\label{Hilbert} (T_{\sigma}f)(b,a,\v)=\Delta(b,a,\v)^{-1/2}\sum_{i,j=1}^2\text{Tr}\Big(D_{i}^j\sigma(b,a,\v,\pi_{i}^j)\widehat{f}(\pi_{i}^j)\pi_{i}^j(b,a,\v)^{\ast}\Big),
\end{equation} for all $(b,a,\v)\in\P$. Now using the fact that $T$ is an unitary operator and Parseval identity for $\F$, we obtain	
\begin{align*}
&\text{Tr}\Big(D_{i}^j\sigma(b,a,\v,\pi_{i}^j)\widehat{f}(\pi_{i}^j)\pi_{i}^j(b,a,\v)^{\ast}\Big)\\
&=\text{Tr}\Big(W_{\sigma_{f}}^{i,j}W_{\tau_{\alpha(b,a,\v)}}\Big)\\
&= 	\int_{\R^2}\int_{\R^2} \sigma_{f}(x,{\xi})\tau_{\alpha(b,a,\v)}(x,{\xi})dxd\xi\\
&=(2\pi)\int_{\R^2}\int_{\R^2}\F_2TK_{i,j}^f(x,\tilde{\xi})\F_2^{-1}TK_{\alpha(b,a,\v)}(x,\tilde{\xi})dxd\xi\\
& =(2\pi)\int_{\R^2}\int_{\R^2}K_{i,j}^f(x,{\xi})K_{\alpha(b,a,\v)}(x,{\xi})dxd\xi\\
&=(2\pi) \int_{C_i^j}\int_{C_{i}^j} \dfrac{|\langle x,x\rangle|}{|\langle \xi;\xi\rangle|^2}(\F_1f)\Bigg(x,\Bigg(\dfrac{\langle \xi;\xi\rangle}{\langle x;x\rangle}\Bigg)^{\frac{1}{2}}, \cosh^{-1}\Bigg(\dfrac{\langle x;\xi\rangle}{ \langle x;x\rangle}\times \Bigg(\dfrac{\langle x;x\rangle}{\langle \xi;\xi\rangle}\Bigg)^{\frac{1}{2}}\Bigg)\\
&\times (\F_1^{-1}\alpha(b,a,\v))\Bigg(x,\Bigg(\dfrac{\langle \xi;\xi\rangle}{\langle x;x\rangle}\Bigg)^{\frac{1}{2}}, \cosh^{-1}\Bigg(\dfrac{\langle x;\xi\rangle}{ \langle x;x\rangle}\times \Bigg(\dfrac{\langle x;x\rangle}{\langle \xi;\xi\rangle}\Bigg)^{\frac{1}{2}}\Bigg)dxd\xi\\
&=(2\pi)\int_{C_{i}^j}\int_0^{\infty}\int_{\R} (\mathcal{F}_1f)(x,c,t)(\F_1^{-1}\alpha(b,a,\v))(x,c,t)\dfrac{dxdcdt}{c^{3}},
\end{align*}
for all $(b,a,\v)\P$. Using the Parseval identity for $\F$, the equation \eqref{Hilbert} becomes 
\begin{align*}
	(T_{\sigma}f)(b,a,\v)&=\dfrac{2\pi}{a}\sum_{i,j=1}^2\int_{C_{i}^j}\int_0^{\infty}\int_{\R}(\mathcal{F}_1f)(x,c,t)(\F_1^{-1}\alpha(b,a,\v))(x,c,t)\dfrac{dxdcdt}{c^{ 3}}\\
	&=\dfrac{2\pi}{a}\int_{\R^2}\int_0^{\infty}\int_{\R}(\mathcal{F}_1f)(x,c,t)(\F_1^{-1}\alpha(b,a,\v))(x,c,t)\dfrac{dxdcdt}{c^{ 3}}\\
	&=\dfrac{2\pi}{a}\int_{\R^2}\int_0^{\infty}\int_{\R}f(x,c,t)\alpha(b,a,\v)(x,c,t)\dfrac{dxdcdt}{c^{3}}.
\end{align*}
So the kernel $k$ of $T_{\sigma}$ is given by
$$k(b,a,\v;x,a^{\prime},\v^{\prime})=\alpha (b,a,\v)(x,a^{\prime},\v^{\prime}).$$ Now
\begin{align*}
	\int_{\P}\int_{\P}|k(b,a,\v;x,c,t)|^2\dfrac{dbdad\v}{a^3}\dfrac{dxdcdt}{c^{3}}
	&= \int_{\P}\int_{\P}|\alpha (b,a,\v)(x,c,t)|^2\dfrac{dbdad\v}{a^3}\dfrac{dxdcdt}{c^{3}}\\
	 &=\int_{\P}\|\alpha(b,a,\v)\|_{L^2(\P)}^2\dfrac{dbdad\v}{a^3}<\infty.
\end{align*}
Thus $T_{\sigma}:L^2(\P)\to L^2(\P)$ is a Hilbert-Schmidt pseudo-differential operator.

 Conversely, let $T_{\sigma}$ be a Hilbert-Schmidt pseudo-differential operator from $L^2(\P)$ to $L^2(\P)$. Then $$(T_{\sigma}f)(b,a,\v)=\int_{\P}\beta(b,a,\v,x,c,t)f(x,c,t)\dfrac{dxdcdt}{c^3},$$ where the kernel  $\beta$ is in $L^2(\P\times\P)$. Let $\alpha:\P\to L^2(\P)$ be the mapping define by
$$\alpha(b,a,\v)(x,c,t)=\beta(b,a,\v,x,c,t).$$
Now reversing the arguments in the proof of the sufficiency part, we get
\begin{equation}\label{eq1}
	(T_{\sigma}f)(b,a,\v)=\dfrac{1}{a}\sum_{i,j=1}^2\text{Tr}[W_{\sigma_{f}}^{i,j}W_{\tau_{\alpha(b,a,\v)}}],
\end{equation}
with $\tau_{\alpha_{(b,a,\v)}}(x,{\xi})=\F_2^{-1}TK_{\alpha_{(b,a,\v)}}(x,\tilde{\xi}), (x,\xi)\in C_i^j\times C_i^j.$ But any pseudo-differential operator on $\P$ is of the form 
\begin{equation}\label{eq2}
(T_{\sigma}f)(b,a,\v)=\Delta(b,a,\v)^{-1/2}\sum_{i,j=1}^2\text{Tr}\Big(D_{i}^j\sigma(b,a,\v,\pi_{i}^j)\widehat{f}(\pi_{i}^j)\pi_{i}^j(b,a,\v)^{\ast}\Big),	
\end{equation}
and $f(\pi_i^j)=W_{\sigma_{f}}^{i,j}, i,j=1,2.$
By equating both equations \eqref{eq1} and \eqref{eq2}, we get
$$\dfrac{1}{a}\sum_{i,j=1}^2\text{Tr}\Big[\big(\pi_i^j(b,a,\v)^{\ast}D_i^j\sigma(b,a,\v,\pi_i^j)-W_{\tau_{\alpha_{(b,a,\v)}}}\big)\widehat{f}(\pi_i^j)\Big]=0,$$ for all $(b,a,\v)$ in $\P$. Hence by uniqueness Theorem \ref{uniqueness1}, we obtain
$$\pi_i^j(b,a,\v)^{\ast}D_i^j\sigma(b,a,\v,\pi_i^j)=W_{\tau_{\alpha_{(b,a,\v)}}}.$$
\end{pff}

An immediate consequence of Theorem \ref{Hilbert-Schmidt}, we have the following result.
\begin{Cor}
If $\alpha:\P\to L^2(\P)$ such that 
$$\int_{\P}|\alpha(b,a,\v,b,a,\v)|\dfrac{dbdad\v}{a^3}<\infty.$$ Let $\sigma:\P\times\P\to S_2$ be a symbol as in Theorem \ref{Hilbert-Schmidt}. Then $T_{\sigma}:L^2(\P)\to L^2(\P)$ is a trace class operator and 
$$\operatorname{Tr}(T_{\sigma})=\int_{\P}\alpha(b,a,\v,b,a,\v)\dfrac{dbdad\v}{a^3}.$$ 
\end{Cor}
\begin{theorem}[Trace class pseudo-differential operator]
Let $\sigma:\P\times\widehat{\P}\to S_2$ be a symbol satisfying the hypotheses of Theorem \ref{uniqueness}. Then the pseudo-differential operator $T_{\sigma}:L^2(\P)\to L^2(\P)$ is a trace-class operator iff
 $$D_i^j\sigma(b,a,\v,\pi_i^j)\pi_i^j(b,a,\v)^{\ast}=W_{\tau_{\alpha(b,a,\v)}},\quad i,j=1,2,$$ for all $(b,a,\v)\in\P$,
where $\alpha:\P\to L^2(\P)$ is a mapping such that the conditions of Theorem \ref{Hilbert-Schmidt} are satisfied and $$\alpha(b,a,\v)(x,c,t)=\int_{\P}k_1(b,a,\v)(b^{\prime},a^{\prime},\v^{\prime})k_2(b^{\prime},a^{\prime},\v^{\prime})(x,c,t)\dfrac{db^{\prime}da^{\prime}d\v^{\prime}}{a^{\prime 3}},$$	for all $(b,a,\v),(x,c,t)\in\P,$ and $ k_1:\P\to L^2(\P), k_2:\P\to L^2(\P)$ satisfy $$\int_{\P}\|k_i(b,a,\v)\|_{L^2(\P)}^2\dfrac{dbdad\v}{a^3}<\infty ,i=1,2.$$ Furthermore, if $T_{\sigma}:L^2(\P)\to L^2(\P)$ is a trace-class operator, then $$\operatorname{Tr}(T_{\sigma})=\int_{\P}\alpha(b,a,\v)(b,a,\v)\dfrac{dbdad\v}{a^3}.$$ 
\end{theorem}
\begin{pff}
The one way of the proof follows from Theorem \ref{Hilbert-Schmidt}. The converse part is proved using the fact that every trace class operator is product of two Hilbert-Schmidt operators from \cite{brislawn}.  
\end{pff}

	\section{Fourier-Wigner Transform and Weyl Transform on Affine Poincar\'e Group}\label{s5}
Let $(b,a,\v)\in\P$, then we find  $(b^{\prime},a^{\prime},\v^{\prime})\in\P$ such that $$(b^{\prime},a^{\prime},\v^{\prime})\ast (b^{\prime},a^{\prime},\v^{\prime})=(b,a,\v),$$ which is $$(b^{\prime},a^{\prime},\v^{\prime})=\Big(\dfrac{b+\sqrt{a}\Lambda_{-\v/2}b}{1+a^2+2\sqrt{a}\cosh\v/2},\sqrt{a},\v/2\Big).$$	We denote 
$$(b,a,\v)^{1/2}=\Big(\dfrac{b+\sqrt{a}\Lambda_{-\v/2}b}{1+a^2+2\sqrt{a}\cosh\v/2},\sqrt{a},\v/2\Big).$$ With this construction, we define the Fourier-Wigner transform on $\P$. Let $f,g\in L^2(\P)$. Then the  Fourier-Wigner transform $V(f,g)$ of $f$ and $g$ is on $\widehat{\P}\times\P$ is defined by
$$(V(f,g)(\pi_i^j,\xi)\varphi)(x)=\int_{\P}f(\xi^{1/2}\ast w)\overline{g(w^{-1}\ast\xi^{1/2})}(\pi_i^j(w)D_{i}^j\varphi)(x)d\mu(w),$$
for all $\xi\in\P,\varphi\in L^2(C_i^j)$, where $i,j=1,2$. Let us denote the function $F^{\xi}$ on $\P$ by $$F^{\xi}(w)=f(\xi^{1/2}\ast w)\overline{g(w^{-1}\ast\xi^{1/2})}.$$ Then $$V(f,g)(\pi_i^j,\xi)=(\mathcal{F}_{\P}F^{\xi})(\pi_i^j),\quad i,j=1,2.$$
Now we define another transform, known as Wigner transform on the group $\P$. Let $f,g\in L^2(\P)$. Then the Wigner transform $W(f,g)$ of $f$ and $g$ on ${\P}\times\widehat{\P}$ is defined by	
$$W(f,g)(w,\pi_i^j)=\Big(\mathcal{F}_{\P,2}\mathcal{F}_{\P,1}^{-1}V(f,g)\Big)(w,\pi_i^j),$$ where $\mathcal{F}_{\P,2}$ is the Fourier transform of $V(f,g)$ in second variable and $\mathcal{F}_{\P,1}^{-1}$ is the inverse Fourier transform of $V(f,g)$ in first variable. After some simple calculation we obtain
\begin{equation}\label{Wigner}
	(W(f,g)(w,\pi_i^j)\varphi)(x)=\int_{\P}f(\xi^{1/2}\ast w)\overline{g(w^{-1}\ast\xi^{1/2})}\big(\pi_i^j(\xi)D_{i}^j\varphi\big)(x)d\mu(\xi),
\end{equation}
for all $w\in\P,\varphi\in L^2(C_i^j)$, where $i,j=1,2$.\\
Let us define $L^2(\widehat{\P}\times\P,S_2)$ be the space of all operator valued functions $F$ defined on $\widehat{\P}\times\P$ to $S_2$ such that $$\|F\|^2_{L^2(\widehat{\P}\times\P,S_2)}\coloneqq\int_{\P}\|F(\pi_i^j,b,a,\v)\|_{S_2(C_i^j)}^2d\mu(b,a,\v)<\infty,\quad i,j=1,2,$$ 
 and the inner product is defined for every $F$ and $G$ in $L^2(\widehat{\P}\times\P,S_2)$ by 
$$ \langle F,G\rangle_{L^2(\widehat{\P}\times\P,S_2)}\coloneqq\int_{\P}\sum_{i,j=1}^2\text{Tr}\Big[F(\pi_i^j,b,a,\v)G(\pi_i^j,b,a,\v)^{\ast}\Big]d\mu(b,a,\v).$$
Next result is related to the orthogonality relation of Fourier-Wigner transform.
\begin{theorem}\label{Fourier-Wigner}
Let $f,g,h$, and $k$ be in $L^2(\P)$. Then
$$\langle V(f,h),V(g,k)\rangle_{L^2(\widehat{\P}\times\P,S_2)}=\langle f,g\rangle_{L^2(\P)}\overline{\langle h,k\rangle_{L^2(\P)}}.$$	
\end{theorem} 
\begin{pff}
 Let $$F^{\xi}(w)=f(\xi^{1/2}\ast w)\overline{h(w^{-1}\ast\xi^{1/2})}, G^{\xi}(w)=g(\xi^{1/2}\ast w)\overline{k(w^{-1}\ast\xi^{1/2})}.$$   Then $$V(f,h)(\pi_i^j,\xi)=(\mathcal{F}_{\P}F^{\xi})(\pi_i^j),V(g,k)(\pi_i^j,\xi)=(\mathcal{F}_{\P}G^{\xi})(\pi_i^j)\quad i,j=1,2.$$ Now using the Plancherel formula for group Fourier transform, we get
\begin{align}\label{eqFW}
&\langle V(f,h),V(g,k)\rangle_{L^2(\widehat{\P}\times\P,S_2)}\\
&=\int_{\P}\sum_{i,j=1}^2\text{Tr}\Big[V(f,h)(\pi_i^j,b,a,\v)V(g,k)(\pi_i^j,b,a,\v)^{\ast}\Big]d\mu(b,a,\v)\nonumber\\
&=\int_{\P}\sum_{i,j=1}^2\text{Tr}\Big[\F_{\P}F^{\xi}(\pi_i^j)\F_{\P}G^{\xi}(\pi_i^j)^{\ast}\Big]d\mu(b,a,\v)\nonumber\\
&= \int_{\P}\langle (\mathcal{F}_{\P}F^{\xi}),(\mathcal{F}_{\P}G^{\xi})\rangle_{L^2(\widehat{\P},S_2)}d\mu(b,a,\v)\nonumber \\
&=\int_{\P}\langle F^{\xi},G^{\xi}\rangle_{L^2(\P)}	d\mu(b,a,\v)\nonumber\\
&=\int_{\P}\int_{\P}f(\xi^{\frac{1}{2}}\ast w)\overline{h(w^{-1}\ast\xi^{\frac{1}{2}})}\overline{g(\xi^{\frac{1}{2}}\ast w)}{k(w^{-1}\ast\xi^{\frac{1}{2}})}d\mu(\xi)d\mu(w).
\end{align}
We make a change of variable from $w=(x,c,t)$ to $w^{\prime}=(x^{\prime},c^{\prime},t^{\prime}),$ by $w^{\prime}=\xi^{\frac{1}{2}}\ast w$. Then using the properties of left-invariant Haar measure, the equation \eqref{eqFW} becomes
\begin{align*}
	\langle V(f,h),V(g,k)\rangle_{L^2(\widehat{\P}\times\P,S_2)}&=\int_{\P}\int_{\P}f(w^{\prime})\overline{h(w^{\prime -1}\ast\xi)}\overline{g(w^{\prime})}{k(w^{\prime -1}\ast\xi)}d\mu(\xi)d\mu(w^{\prime})\\
	&=\langle f,g\rangle_{L^2(\P)}\langle k,h\rangle_{L^2(\P)}.
\end{align*}
Hence this complete the proof. 	
\end{pff}

Similarly, let us define $L^2(\P\times\widehat{\P},S_2)$ be the space of all operator valued functions $F$ defined on $\P\times\widehat{\P}$ to $S_2$ such that $$\|F\|^2_{L^2(\P\times\widehat{\P},S_2(C_i^j))}\coloneqq\int_{\P}\|F(b,a,\v,\pi_i^j)\|_{S_2(C_i^j)}^2d\mu(b,a,\v)<\infty,\quad i,j=1,2,$$ 
and the inner product is defined for every $F$ and $G$ in $L^2(\widehat{\P}\times\P)$ by 
$$ \langle F,G\rangle_{L^2(\P\times\widehat{\P})}\coloneqq\int_{\P}\sum_{i,j=1}^2\text{Tr}\Big[F(b,a,\v,\pi_i^j)G(b,a,\v,\pi_i^j)^{\ast}\Big]d\mu(b,a,\v).$$
Next we prove the orthogonality relation of Wigner transform.
\begin{theorem}
Let $f,g,h$, and $k$ be in $L^2(\P)$. Then
\begin{equation}\label{Wignerorthogonal}
	\langle W(f,h),W(g,k)\rangle_{L^2(\P\times\widehat{\P},S_2)}=\langle f,g\rangle_{L^2(\P)}\overline{\langle h,k\rangle_{L^2(\P)}}.
\end{equation}	
\end{theorem} 
\begin{pff}
The proof follows directly from Theorem \ref{Fourier-Wigner}.	
\end{pff}	
Let $\sigma:\P\times\widehat{\P}\to B(L^2(C_i^j))$ be an operator valued symbol. Then the Weyl transform $W_{\sigma}$ associated to an operator valued symbol $\sigma$ is defined by
\begin{equation}\label{Weyl}
	\langle W_{\sigma}f,g\rangle_{L^2(\P)}=\int_{\P}\sum_{i,j=1}^2\text{Tr}\Big[D_i^j\sigma(b,a,\v,\pi_i^j)W(f,g)(b,a,\v,\pi_i^j)\Big]\dfrac{dbdad\v}{a^3},
\end{equation}	
for all $f,g\in L^2(\P)$, where $W(f,g)$ is the Wigner transform of $f$ and $g$ defined in \eqref{Wigner}.	
We prove the Weyl transform defined in \eqref{Weyl} is bounded on $L^2(\P)$.
\begin{theorem}
	Let $\sigma:\P\times\widehat{\P}\to S_2$ be an operator-valued symbol such that the mappings
	$$\P\times\widehat{\P}\owns (b,a,\v,\pi_i^j)\mapsto D_i^j\sigma(b,a,\v,\pi_i^j)\in S_2(C_i^j)$$ are satisfying 
	\begin{equation}\label{Wbounded}
		\int_{\P}\|D_i^j\sigma(b,a,\v,\pi_i^j)\|_{S_2(C_i^j)}^2\dfrac{dbdad\v}{a^3}<\infty,i,j=1,2.
	\end{equation} Then $W_{\sigma}:L^2(\P)\to L^2(\P)$ is a bounded operator.
\end{theorem}	
\begin{pff}
Let $f,g$ be in $L^2(\P)$. Then by the definition of Weyl transform \eqref{Weyl}, Schwartz inequality, and equations \eqref{Wbounded} and \eqref{Wignerorthogonal}, we get	
\begin{align*}
	|\langle W_{\sigma}f,g\rangle_{L^2(\P)}| &\leq\int_{\P}\sum_{i,j=1}^2\Bigl|\text{Tr}\Big[D_i^j\sigma(b,a,\v,\pi_i^j)W(f,g)(b,a,\v,\pi_i^j)\Big]\Bigr|d\mu(b,a,\v)\\
	&\leq \int_{\P}\sum_{i,j=1}^2\|D_i^j\sigma(b,a,\v,\pi_i^j)\|_{S_2(C_i^j)}\|W(f,g)(b,a,\v,\pi_i^j)\|_{S_2(C_i^j)}d\mu(b,a,\v)\\
	&\leq \Bigg(\int_{\P}\sum_{i,j=1}^2\|D_i^j\sigma(b,a,\v,\pi_i^j)\|_{S_2(C_i^j)}^2d\mu(b,a,\v)\Bigg)^{\frac{1}{2}}\\
	&\times\Bigg(\int_{\P}\sum_{i,j=1}^2 |W(f,g)(b,a,\v,\pi_i^j)\|_{S_2(C_i^j)}d\mu(b,a,\v)\Bigg)^{\frac{1}{2}}\\
	&= \|D\sigma\|_{L^2(\P\times\widehat{\P},S_2)}\times \|W(f,g)\|_{L^2(\P\times\widehat{\P},S_2)}\\
	&=\|D\sigma\|_{L^2(\P\times\widehat{\P},S_2)}\|f\|_{L^2(\P)}\|g\|_{L^2(\P)},
\end{align*}
where $$\|D\sigma\|_{L^2(\P\times\widehat{\P},S_2)}^2=\int_{\P}\sum_{i,j=1}^2\|D_i^j\sigma(b,a,\v,\pi_i^j)\|_{S_2(C_i^j)}^2d\mu(b,a,\v).$$
\end{pff}

\end{document}